\documentclass[oneside,english]{amsart}
\usepackage[T1]{fontenc}
\usepackage[latin9]{inputenc}
\usepackage{color}
\usepackage{babel}
\usepackage{verbatim}
\usepackage{units}
\usepackage{amstext}
\usepackage{amsthm}
\usepackage{amssymb}
\usepackage[unicode=true,pdfusetitle,
 bookmarks=true,bookmarksnumbered=false,bookmarksopen=false,
 breaklinks=false,pdfborder={0 0 0},pdfborderstyle={},backref=false,colorlinks=true]
 {hyperref}

\makeatletter

\newcommand{\noun}[1]{\textsc{#1}}

\numberwithin{equation}{section}
\numberwithin{figure}{section}
\theoremstyle{plain}
\newtheorem{thm}{\protect\theoremname}
  \theoremstyle{plain}
  \newtheorem{lem}[thm]{\protect\lemmaname}
  \theoremstyle{definition}
  \newtheorem{defn}[thm]{\protect\definitionname}
  \theoremstyle{plain}
  \newtheorem{prop}[thm]{\protect\propositionname}
  \theoremstyle{remark}
  \newtheorem*{rem*}{\protect\remarkname}
  \theoremstyle{plain}
  \newtheorem{cor}[thm]{\protect\corollaryname}
  \theoremstyle{definition}
  \newtheorem{problem}[thm]{\protect\problemname}

\makeatother

  \providecommand{\corollaryname}{Corollary}
  \providecommand{\definitionname}{Definition}
  \providecommand{\lemmaname}{Lemma}
  \providecommand{\problemname}{Problem}
  \providecommand{\propositionname}{Proposition}
  \providecommand{\remarkname}{Remark}
\providecommand{\theoremname}{Theorem}

\begin{document}

\title[Symmetric orthogonality and non-expansive projections]{Symmetric orthogonality and non-expansive projections in metric spaces}

\author{Martin Kell}

\address{Mathematisches Institut, Universität Tübingen, Tübingen, Germany}

\email{martin.kell@math.uni-tuebingen.de}
\begin{abstract}
In this paper known results of symmetric orthogonality, as introduced
by G. Birkhoff, and non-expansive nearest point projections are extended
from the linear to the metric setting. If the space has non-positive
curvature in the sense Busemann then it is shown that those concepts
are actually equivalent. In the end it is shown that every space having
non-positive curvature in the sense of Busemann is a $CAT(0)$-space
provided that its tangent cones are uniquely geodesic and their nearest
point projections onto convex are non-expansive. 
\end{abstract}

\maketitle
\global\long\def\Geo{\operatorname{Geo}}
\global\long\def\conv{\operatorname{conv}}

Orthogonality in the Euclidean setting can be described either by
an angle condition or by a nearest point projection property using
Pythagoras' theorem. More precisely, one says a geodesic $\gamma$
intersects a geodesic $\eta$ orthogonally if the intersection point
$p$ is the closest point of $\gamma_{1}$ on $\eta$. We denote this
by $\gamma\bot_{p}\eta$. It is not difficult to see that in the Euclidean
setting $\gamma$ intersects $\eta$ orthogonally if and only if $\eta$
intersects $\gamma$ orthogonality. Hence Euclidean spaces are said
to have the \emph{symmetric orthogonality property} $(SO)$. 

For general normed spaces the orthogonality condition in terms of
projections appeared the first time in Birkhoff's work \cite[Theorem 2]{Birkhoff1935}.
He showed that the only higher dimensional Banach spaces having symmetric
orthogonalities are Hilbert spaces. Later the symmetric orthogonality
property appeared again in the setting of Hilbert geometries \cite{Kelly1952}
where it was called symmetric perpendicularity. More recently, it
appeared in the metric setting under the name of property $(B)$ in
\cite{Kuwae2013}.

An essential ingredient to obtain the non-expansive behavior for gradient
flows of convex functionals in spaces having non-positive curvature
in the sense of Alexandrov is the fact that nearest point projections
onto convex sets are themselves non-expansive. This result is well-known
for Hilbert spaces. In fact, Kakutani \cite[Therem 3]{Kakutani1939}
showed that Hilbert spaces are the only higher dimensional Banach
spaces having \emph{non-expansive projections} onto convex sets. This
result was rediscovered later by Phelps \cite[Theorem 5.2]{Phelps1957}.
Subsequently, it was generalized to Hadamard manifolds and $CAT(0)$-spaces,
see bibliographic remarks of \cite[Chapter 2 \& 5]{Bacak2014a}. 

The non-expansive projection property $(NE)$ can be also used to
prove a \emph{convexity principle}, a kind of generalized maximum
principle for harmonic maps into non-positively curved spaces. More
precisely, if $h:M\to N$ is harmonic, i.e. a local minimizer of the
Dirichlet energy, then $h(M)$ is contained in the closed convex hull
of $h(\partial M)$. Indeed, this follows from the non-expansive projection
property and the fact that the Dirichlet energy $E$ satisfies $E(\pi_{C}\circ f)\le E(f)$
whenever $f\big|_{\partial M}=h\big|_{\partial M}$, see \cite[Chapter 8]{Jost1994}.
Here $\pi_{C}$ denotes the nearest point projection onto the convex
hull $C$, i.e. 
\[
\pi_{C}(z)=\{z_{C}\in C\,|\,d(z,z_{C})=\inf_{z'\in C}d(z,z')\}.
\]
Using projection one readily verifies that $\gamma\bot_{p}\eta$ holds
whenever $\gamma$ and $\eta$ intersect in $p$ and satisfy $p\in\pi_{\eta}(\gamma_{1})$. 

In this note we show that the concepts of non-expansive projections
$(NE)$ and symmetric orthogonalities $(SO)$ are strongly related.
Indeed, Proposition \ref{prop:CPtoSO} shows that having non-expansive
projections implies the symmetric orthogonality property. Note that
the opposite is wrong as can be observed on a small convex domain
of the sphere $\mathbb{S}^{n}$. 

Assuming a form of non-positive curvature assumption introduced by
Busemann we are able to prove the following equivalence between the
two conditions.
\begin{thm}
[see Lemma \ref{lem:SOiffA} and Theorem \ref{thm:NEiffSO}]\label{thm:Main1-SOeqNE}Assume
$(M,d)$ is a complete Busemann convex geodesic space, i.e. for all
geodesics $\gamma$ and $\eta$ with $\gamma_{0}=\eta_{0}$ it holds
\[
d(\gamma_{\frac{1}{2}},\eta_{\frac{1}{2}})\le\frac{1}{2}d(\gamma_{1},\eta_{1}).
\]
Then the following properties are equivalent:
\begin{itemize}
\item \noun{(Property $(SO)$)} For all geodesics $\gamma$ and $\eta$
with $\gamma_{0}=\eta_{0}$ it holds 
\[
\gamma\bot_{\gamma_{0}}\eta\,\Longleftrightarrow\,\eta\bot_{\gamma_{0}}\gamma.
\]
\item \noun{(Property $(A)$)} Whenever $\gamma$ is geodesic and $x\in M$
then for all $t\in[0,1]$ and $x_{C}\in\pi_{\gamma}(x)$ it holds
\[
d(x_{C},\gamma_{t})\le d(x,\gamma_{t}).
\]
\item \noun{(Property $(NE)$)} Whenever $C$ is a closed convex set and
$x,y\in M$ then for all $x_{C}\in\pi_{C}(x)$ and $y_{C}\in\pi_{C}(y)$
it holds 
\[
d(x_{C},y_{C})\le d(x,y).
\]
 
\end{itemize}
\end{thm}
In the following $M\times_{2}\mathbb{R}$ denotes the metric space
$(M\times\mathbb{R},\tilde{d})$ equipped with the metric $\tilde{d}((x,t),(y,s))^{2}=d(x,y)^{2}+|t-s|^{2}$.
Though the properties of the theorem above suggest that the space
in question has well-defined angles, this turns out to be wrong. Indeed,
Birkhoff showed there is an abundance of strictly convex norms on
$\mathbb{R}^{2}$ having symmetric orthogonalities \cite{Birkhoff1935}. 

Observe, however, that whenever $\mathbb{R}^{2}\times_{2}\mathbb{R}$
has symmetric orthogonalities for a normed space $(\mathbb{R}^{2},\|\cdot\|)$
then by Birkhoff's result $\mathbb{R}^{2}\times_{2}\mathbb{R}$ and
thus $\mathbb{R}^{2}$ must be Euclidean. Hence the symmetric orthogonality
property is not stable under taking products. If we can ensure that
tangent spaces are still uniquely geodesic then we obtain the following
characterization by assuming a stable version of either the symmetric
orthogonality property or the non-expansive projection property. 
\begin{thm}
[see Theorem \ref{thm:BusemannSO-CAT0}]\label{thm:Main2-BCplusSOequivCAT0}Assume
$(M,d)$ and all its tangent space $(T_{x}^{(o)}M,d_{x})$ are Busemann
convex geodesic spaces. If $M\times_{2}\mathbb{R}$ satisfies either
of the properties $(SO)$, $(A)$ or $(NE)$ then $(M,d)$ is a $CAT(0)$-space. 
\end{thm}
Note that any $CAT(0)$-space $(M,d)$ satisfies the assumptions of
Theorem \ref{thm:Main2-BCplusSOequivCAT0} implying that the stable
version of the properties $(SO)$, $(A)$ and resp. $(NE)$ are Riemannian
properties. Indeed, even without Busemann convexity Proposition \ref{prop:two-sided-angles}
shows that a weak form of angles exists assuming the stable symmetric
orthogonality property and a weak form of convexity of the metric.

\subsection*{Preliminaries}

A metric space $(M,d)$ is said to be a \emph{geodesic space} if for
each $x,y\in M$ there is a $1$-Lipschitz map $\gamma:[0,1]\to M$
such that $\gamma_{0}=x$, $\gamma_{1}=y$ and 
\[
d(\gamma_{t},\gamma_{s})=|s-t|d(x,y).
\]
The map $\gamma$ is called a \emph{$[0,1]$-parametrized geodesic}.
$(M,d)$ is said to be \emph{uniquely geodesic} if for each $x,y\in M$
there is exactly one geodesic connecting those points. Similarly,
a $[a,b]$-parametrized geodesic $\gamma:[a,b]\to M$ if $\tilde{\gamma}_{t}=\gamma_{a+t(b-a)}$
is a $[0,1]$-parametrized geodesic. A unit speed geodesic between
$x$ and $y$ is a $[0,d(x,y)]$-parametrized geodesic between $x$
and $y$. Finally, we say $\eta:[0,\infty)\to M$ is a geodesic ray
and $\gamma:\mathbb{R}\to M$ is a geodesic line if for all $t,s\in[0,\infty)$
and $t',s'\in\mathbb{R}$ it holds $d(\eta_{t},\eta_{s})=|t-s|$ and
$d(\gamma_{t'},\gamma_{s'})=|t'-s'|$. Note that a geodesic line $\gamma$
induces two geodesic rays $\gamma^{\pm}$ defined by $\gamma_{t}^{\pm}=\gamma_{\pm t}$,
$t\in[0,\infty)$.

In the following, we always assume that metric spaces are complete
and geodesic and, if not mentioned otherwise, geodesics are assumed
to be $[0,1]$-parametrized.

A subset $C\subset M$ is \emph{weakly convex} if for all $x,y\in C$
there is a geodesic $\gamma$ connecting $x$ and $y$ such that $\gamma_{t}\in C$
for $t\in[0,1]$. If all geodesics connecting $x$ and $y$ lie entirely
in $C$ then $C$ is said to be \emph{convex}. Note that the image
of any geodesic is weakly convex. 

Given any subset $A\subset M$ we define the convex hull of $A$ as
follows: Let $G_{0}=A$ and for $n\ge1$ define 
\[
G_{n}=\bigcup_{x,y\in G_{n-1}}\{\gamma_{t}\,|\,\gamma\mbox{ is a geodesic connecting \ensuremath{x\:}and \ensuremath{y,\:}and \ensuremath{t\in[0,1]\}}}
\]
\[
\conv A=\bigcup_{n\in\mathbb{N}}G_{n}.
\]
 The closed convex hull of $A$ is now defined to be the closure of
$\conv A$.

Given a subset $C\subset M$ and $x\in M$, denote by $\pi_{C}(x)$
the set of nearest points of $x$ onto $C$, i.e. 
\[
\pi_{C}(x)=\{y\in C\,|\,d(x,y)=d(x,C)=\inf_{z\in C}d(x,z)\}.
\]
If $C$ is compact then $\pi_{C}(x)$ is always non-empty. If $\pi_{C}$
is single-valued then we regard $\pi_{C}$ as a (partially-defined)
map. In general $\pi_{C}(x)$ is neither non-empty nor single-valued. 

We say $(M,d)$ is \emph{uniformly $\infty$-convex} if there is a
function $\rho:(0,\infty)\to(0,\infty)$ such that for all $\epsilon>0$
and all points $x,y,z\in M$ with 
\[
d(y,z)>\epsilon\max\{d(x,y),d(x,z)\}
\]
it holds 
\[
d(x,m)\le(1-\rho(\epsilon))\max\{d(x,y),d(x,z)\}
\]
whenever $m$ is a midpoint of $y$ and $z$. Note that in a uniformly
$\infty$-convex space the projections $\pi_{C}$ onto any closed
convex set is single-valued. 

This definition of uniform convexity is equivalent to uniform convexity
in Banach spaces. Among several uniform convexity assumptions in the
metric setting this is one of the weakest (see e.g. \cite{Foertsch2004,Kell2014}).
This condition is needed to ensure that the tangent spaces in terms
of ultralimit blow-ups are uniquely geodesic (compare Lemma \ref{lem:ultralimit-so}
and \ref{lem:ultralimit-uni-conv} below). Every Riemannian and Finsler
manifold is locally uniformly $\infty$-convexity as this condition
is equivalent to strong convexity of balls in the smooth finite dimensional
setting (see \cite[Theorem 5.2]{Shen1997}). 

A popular and stronger condition is called \emph{Busemann convexity}
or \emph{non-positive curvature in the sense of Busemann} \cite[Section 36]{Busemann1955}.
For this one requires that for any geodesics $\gamma$ and $\eta$
the map 
\[
t\mapsto d(\gamma_{t},\eta_{t})
\]
is convex. As one may readily verify the condition is equivalent to
the following: for all geodesics $\gamma$ and $\eta$ with $\gamma_{0}=\eta_{0}$
it holds 
\[
d(\gamma_{\frac{1}{2}},\eta_{\frac{1}{2}})\le\frac{1}{2}d(\gamma_{1},\eta_{1}).
\]
Note that Busemann convex spaces are uniquely geodesic and the projection
onto compact convex sets, e.g. geodesics, is well-defined and single-valued.
Busemann's condition also implies the following rigidity theorem.
\begin{lem}
[{\cite[Theorem (36.9)]{Busemann1955}}]\label{lem:Busemann-convex-affine-rigidity}Let
$(M,d)$ be Busemann convex metric space. Assume there are two geodesics
$\gamma$ and $\eta$ such that $t\mapsto d(\gamma_{t},\eta_{t})$
is affine, i.e. 
\[
d(\gamma_{t},\eta_{t})=(1-t)d(\gamma_{0},\eta_{0})+td(\gamma_{1},\eta_{1}).
\]
Then the closed convex hull of $\gamma([0,1])\cup\eta([0,1])$ is
isometric to a convex subset of $\mathbb{R}^{2}$ equipped with a
strictly convex norm.
\end{lem}
However, Busemann convexity is not stable under taking limit spaces
as can be observed by letting the $n$-dimensional $\ell^{p}$-spaces
converge to $(\mathbb{R}^{n},\|\cdot\|_{\infty})$. Furthermore, it
is not known whether tangent spaces of Busemann convex spaces are
themselves Busemann convex. 

A condition which is stable under limit operations is the $CAT(0)$-condition
which can be formulated via comparison triangles (see e.g. \cite[Chapter II.1]{Bridson1999}).
More precisely, $(M,d)$ is said to be a \emph{$CAT(0)$-space} if
for all $x,y,z\in M$ and $\tilde{x},\tilde{y},\tilde{z}\in\mathbb{R}^{2}$
satisfying 
\[
d(x,y)=\|\tilde{x}-\tilde{y}\|,d(x,z)=\|\tilde{x}-\tilde{z}\|,d(y,z)=\|\tilde{y}-\tilde{z}\|
\]
it holds
\[
d(x,m)\le\|\tilde{x}-\tilde{m}\|
\]
where $m$ and $\tilde{m}$ are the midpoints of $y$ and $z$, resp.
$\tilde{y}$ and $\tilde{z}$. 

In the following we define tangent spaces in terms of ultralimits
of a sequence of blow-ups at a fixed point. For this let $\omega$
be a \emph{non-principle ultrafilter} on $\mathbb{N}$. We will regard
$\omega$ as a finitely additive measure on $\mathbb{N}$ such that
for all $A\subset\mathbb{N}$ it holds $\omega(A)\in\{0,1\}$ and
$\omega(A)=0$ whenever $A$ is finite. Given a sequence $(a_{n})_{n\in\mathbb{N}}$
in $\mathbb{R}\cup\{\pm\infty\}$, the ultrafilter $\omega$ ``selects''
exactly one converging subsequence of $(a_{n})_{n\in\mathbb{N}}$.
We denote the unique limit by $\lim_{\omega}(a_{n})_{n\in\mathbb{N}}$
and called the \emph{ultralimit} of $(a_{n})_{n\in\mathbb{N}}$.

With of the ultrafilter $\omega$ we define \emph{blow-up tangent
spaces} as folows: Fix a sequence $(\lambda_{n})_{n\in\mathbb{N}}$
of positive numbers converging to zero. Then the tangent space $(T_{x}^{(o)}M,d_{x},x)$
at $x$ is the ultralimit of the sequence of pointed metric spaces
$(M,\lambda_{n}^{-1}d,x)_{n\in\mathbb{N}}$ where the points of $(T_{x}^{(o)}M,d_{x},x)$
are just the set of sequences $(x_{n})$ with $\lim_{\omega}\lambda_{n}^{-1}d(x_{n},x)<\infty$
and the metric is given by $d_{x}((x_{n}),(y_{n}))=\lim_{\omega}\lambda_{n}^{-1}d(x_{n},y_{n})$.
More precisely, the ultralimit is the set of equivalence classes since
we identify two sequences $(x_{n})$ and $(y_{n})$ whenever $\lim_{\omega}\lambda_{n}^{-1}d(x_{n},y_{n})=0$.
Note that $(T_{x}^{(o)}M,d_{x})$ depends on both the sequence $(\lambda_{n})_{n\in\mathbb{N}}$
and the ultrafilter $\omega$.

One can show that $(T_{x}^{(o)}M,d_{x})$ is complete and geodesic
whenever $(M,d)$ is geodesic. Indeed, if $(x_{n})$ and $(y_{n})$
are points in $T_{x}^{(o)}M$ and $(\gamma_{n})$ a sequence of geodesics
connecting $x_{n}$ and $y_{n}$ then $(\gamma_{n}(t))$ is in $T_{x}^{(o)}M$
and 
\[
\lim_{\omega}\frac{1}{t}d((x_{n}),(\gamma_{n}(t)))=\lim_{\omega}d((x_{n}),(y_{n}))=\frac{1}{1-t}\lim_{\omega}d((x_{n}),(\gamma_{n}(1-t)))
\]
so that $t\mapsto(\gamma_{n}(t))$ is a geodesic in $(T_{x}^{(o)}M,d_{x})$.
Thus if $(T_{x}^{(o)}M,d_{x})$ is uniquely geodesic then any geodesic
in $T_{x}^{(o)}M$ is given as an ultralimit of geodesics in $M$.

We denote the equivalence class obtained from the constant sequence
$x_{n}=x$ by $\mathbf{0}_{x}$. An \emph{ultraray} $\bar{\gamma}:[0,\infty)\to M$
is a geodesic ray obtained as follow: For a point $y\in M$ let $\gamma$
be a geodesic between $x$ and $y$. Then 
\[
\bar{\gamma}_{s}:=\lim_{\omega}\gamma_{\lambda_{n}s},s\ge0.
\]
Similarly, an \emph{ultraline} $\bar{\gamma}:\mathbb{R}\to M$ is
a geodesic line obtained from a unit geodesic $\gamma:[-a,a]\to M$
with $\gamma_{0}=x$ by setting 
\[
\bar{\gamma}_{s}:=\lim_{\omega}\gamma_{\lambda_{n}s},s\in\mathbb{R}.
\]

Note if $(M,d)$ is Busemann convex then for two ultrarays $\bar{\gamma}$
and $\bar{\eta}$ we have 
\[
d_{x}(\bar{\gamma}_{s},\bar{\eta}_{s'})=\lim_{\lambda\to0}\frac{d(\gamma_{\lambda s},\gamma_{\lambda s'})}{\lambda}.
\]
In particular, the metric is independent of the ultrafilter and the
scaling $(\lambda_{n})_{n\in\mathbb{N}}$. Using Lemma \ref{lem:Busemann-convex-affine-rigidity}
above the following holds.
\begin{lem}
\label{lem:flat-sector}Assume $(M,d)$ and $(T_{x}^{(o)}M,d_{x})$
are Busemann convex. Then the convex hull of any two ultrarays is
flat, i.e. it is isometric to a flat sector in the two-dimensional
plane $\mathbb{R}^{2}$ equipped with a strictly convex norm.
\end{lem}
We also need the following technical lemma.
\begin{lem}
\label{lem:Euclidean-double-sector}Assume $(M,d)$ is Busemann convex,
$\gamma:\mathbb{R}\to M$ is a line and $\eta:[0,\infty)\to M$ a
ray such that $\gamma_{0}=\eta_{0}$. If the convex hulls of $\gamma([0,\infty))\cup\eta([0,\infty))$
and $\gamma((-\infty,0])\cup\eta([0,\infty))$ each span a flat Euclidean
sector then for all $s_{-}<0<s_{+}$ and $t\in[0,\infty)$ it holds
\[
d(\eta_{t},\gamma_{0})^{2}\le(1-\lambda)d(\eta_{t},\gamma_{s_{-}})^{2}+\lambda d(\eta_{t},\gamma_{s_{+}})^{2}-(1-\lambda)\lambda d(\gamma_{s_{-}},\gamma_{s_{+}})^{2}
\]
where $(1-\lambda)s_{-}+\lambda s_{+}=0$ and $\lambda\in(0,1)$. 
\end{lem}
\begin{proof}
Denote the two flat Euclidean sector by $C_{\pm}$. The metric spaces
$(C_{\pm},d\big|_{C_{\pm}\times C_{\pm}})$ are both $CAT(0)$-spaces.
If we glue them along the common ray $\eta$ then we obtain a new
$CAT(0)$-space $(\tilde{C},\tilde{d})$ satisfying the above inequality
for $\tilde{d}$. Observe now that for each $(x,y)\in\{(\eta_{t},\gamma_{0}),(\eta_{t},\gamma_{s_{-}}),(\eta_{t},\gamma_{s_{+}})\}$
it holds $\tilde{d}(x,y)=d(x,y)$ as $x$ and $y$ are both either
in $C_{+}$ or $C_{-}$. Finally, as $\gamma$ is a line in $M$,
it must be a line in $\tilde{C}$ as well, i.e. it holds 
\[
d(\gamma_{s_{-}},\gamma_{s_{+}})=|s_{+}-s_{-}|=\tilde{d}(\gamma_{s_{-}},\gamma_{s_{+}}).
\]
Hence the inequality holds for the metric $d$.
\end{proof}

\section*{Symmetric Orthogonality}

In this section we introduce the concept of symmetric orthogonality
and prove some corollaries which can be deduced directly deduced from
Birkhoff's main result on normed spaces (Lemma \ref{lem:Birkhoff-rigidity}).
\begin{defn}
[Birkhoff orthogonal] A geodesic $\gamma$ is said to be \emph{orthogonal}
to a geodesic $\eta$ if they intersect in a common point $p=\gamma_{0}=\eta_{0}$
and 
\[
\forall t,s\in[0,1]:d(p,\gamma_{t})\le d(\eta_{s},\gamma_{t}).
\]
In other words, for all $t\in[0,1]$ the point $p$ is a closest point
of $\gamma_{t}$ on $\eta$. In this case, we write $\gamma\perp_{p}\eta$. 
\end{defn}
For Riemannian manifolds it is known that if $\gamma$ and $\eta$
are short enough, i.e. $\gamma_{t}$ and $\eta_{t}$, $t\in[0,1]$,
stay sufficiently close to $\gamma_{0}$ then $\gamma\perp_{\gamma_{0}}\eta$
is equivalent to 
\[
g_{\gamma_{0}}(\dot{\gamma}_{0},\dot{\eta}_{0})\le0.
\]
Note that this angle characterization is symmetric. In particular,
for short geodesics $\gamma$ and $\eta$ in a Riemannian manifold
it holds $\eta\perp_{\gamma_{0}}\gamma$ if and only if $\eta\perp_{\gamma_{0}}\gamma$.
If the Riemannian manifold is a Hadamard spaces, i.e. it is simply
connected and has non-positive sectional curvature, then $\gamma\perp_{p_{0}}\eta$
is equivalent to $g_{\gamma_{0}}(\dot{\gamma}_{0},\dot{\eta}_{0})\le0$
for all geodesics $\gamma$, $\eta$. The symmetry ``$\gamma\perp_{\gamma_{0}}\eta$
if and only if $\eta\perp_{\gamma_{0}}\gamma$'' holds actually for
all simply connected Riemannian manifolds without focal points. We
leave the details to the interested reader.
\begin{defn}
[Symmetric Orthogonality] A geodesic space $(M,d)$ is said to satisfy
the \emph{symmetric orthogonality property} $(SO)$ if $\gamma\perp_{p}\eta$
implies $\eta\perp_{p}\gamma$. We say that $(SO)_{loc}$ holds if
each point admits a neighborhood $U$ such that the symmetry holds
for all geodesics lying in $U$.
\end{defn}
Any normed space satisfies $(SO)$ if it satisfies $(SO)_{loc}$.
Furthermore, for reflexive normed spaces whose dual space have strictly
convex norm, property \textbf{$(SO)$} is equivalent to ``$\ell_{v}(w)=0$
iff $\ell_{w}(v)=0$'' where $\ell_{v}$ and $\ell_{w}$ are the
duals of $v$ and resp. $w$ obtained by 
\begin{eqnarray*}
\ell_{v}(w') & = & \lim_{\epsilon\to0}\frac{\|v+\epsilon w'\|^{2}-\|v\|^{2}}{2\epsilon}\\
\ell_{w}(v') & = & \lim_{\epsilon\to0}\frac{\|w+\epsilon v'\|^{2}-\|w\|^{2}}{2\epsilon}.
\end{eqnarray*}
In Finsler geometry, $\ell_{v}$ is equal to $g_{v}(v,\cdot)$ where
$g_{v}$ is called the fundamental tensor at $v$. We refer to \cite{Shen1997,Ohta2008}
for all concepts of Finsler manifolds needed for the discussion below. 
\begin{lem}
[Linear Orthogonal Rigidity \cite{Birkhoff1935,James1947}]\label{lem:Birkhoff-rigidity}If
$(\mathbb{R}^{n},\|\cdot\|)$ is a normed vector space with $n>2$
then $(\mathbb{R}^{n},\|\cdot\|)$ satisfies $(SO)$ if and only if
$(\mathbb{R}^{n},\|\cdot\|)$ is Euclidean, i.e. the norm $\|\cdot\|$
is induced by an inner product. 
\end{lem}
Actually this result holds more general for any Finsler manifolds. 
\begin{prop}
\label{prop:Finsler-SO}Assume $(M,F)$ is a smooth Finsler manifold.
Then $(M,F)$ satisfies $(SO)_{loc}$ if and only if each tangent
space $(T_{x}M,F_{x})$ satisfies $(SO)$. In particular, $(M,F)$
is either $2$-dimensional or a Riemannian manifold. 

In case $(M,F)$ is has non-positive flag curvature then the local
condition $(SO)_{loc}$ implies the global condition $(SO)$.
\end{prop}
\begin{rem*}
The proposition also holds for $L^{2}$-products of Finsler manifolds
which, in general, do not have $C^{2}$-Finsler structures. However,
the first variation and the local strong convexity of the square of
the distance still holds for the factors and hence their product.
\end{rem*}
\begin{proof}
For simplicity assume $F$ is symmetric. The proof can be easily adapted
to the asymmetric case as it only relies on the first variation formula
and local strong convexity of the square of the distance.

Assume for some $x_{0}\in M$ the tangent space $(T_{x_{0}}M,F_{x_{0}})$
does not satisfy property $(SO)$. Denote by $g_{v}$ the fundamental
tensor at $v\in TM$. Since the symmetric orthogonality property $(SO)$
does not hold for $(T_{x_{0}}M,F_{x_{0}})$, there are $v,w\in T_{x_{0}}M$
such that 
\[
g_{v}(v,w)=0\ne g_{w}(w,v).
\]
 Let $\gamma_{v},\gamma_{w}:(-\epsilon,\epsilon)$ be geodesics with
$\dot{\gamma}_{v}(0)=v$ and $\dot{\gamma}_{w}(0)=w$. The first variation
formula yields 
\begin{eqnarray*}
\frac{d}{d\epsilon}d^{2}(\gamma_{v}(-t),\gamma_{w}(\epsilon)) & = & g_{v}(v,w)=0\\
\frac{d}{d\epsilon}d^{2}(\gamma_{v}(\epsilon),\gamma_{w}(-s)) & = & g_{w}(w,v)\ne0.
\end{eqnarray*}
By \cite[Theorem 5.2]{Shen1997} (see also \cite[Corollary 5.2]{Ohta2008}),
in a neighborhood $U$ of $x$ the square of the distance from a fixed
point $x'\in U$ is strongly convex (in $U$) hence any critical point
along a geodesic in $U$ is automatically a global minimum (in $U$).
Thus we see that $\gamma_{v}\bot_{x_{0}}\gamma_{w}$ but $\gamma_{w}\not\bot\gamma_{v}$.
In particular, $(M,F)$ cannot satisfy the symmetric orthogonality
property $(SO)$.

In case $(M,F)$ has non-positive flag curvature, the square of the
distance from fixed points is strictly convex \cite{Egloff1997} (see
also remark after \cite[Corollary 5.2]{Ohta2008}). Therefore, $g_{v}(v,w)=0$
iff $\gamma_{v}\bot_{x_{0}}\gamma_{w}$. 
\end{proof}
A well-known class of $n$-dimensional simply connected Finsler manifolds
with non-positive flag curvature are Hilbert geometries (see \cite{Papadopoulos2014}
for an introduction to Hilbert geometries). One may verify that Hilbert
geometries whose tangent norms satisfy everywhere $(SO)$ must already
be Riemannian manifolds and therefore isometric to the hyperbolic
space. Using a very elegant and short argument this fact was obtained
by Kelly and Paige in \cite{Kelly1952}.
\begin{prop}
[Hyperbolic Orthogonal Rigidity in Hilbert Geometry \cite{Kelly1952}]Any
$n$-dimensional Hilbert geometry satisfying $(SO)$ is isometry the
$n$-dimensional hyperbolic space.
\end{prop}
The general condition $(SO)$ does not exclude all non-Riemannian
geometry. Nevertheless, it indicates that it is not stable, i.e. in
general the $L^{2}$-product of two spaces satisfying $(SO)$ does
not satisfy $(SO)$, not even locally. For that reason we define a
stronger condition. Indeed, if $(\mathbb{R}^{2}\times\mathbb{R},(\|\cdot\|^{2}+|\cdot|)^{\frac{1}{2}})$
satisfies the symmetric orthogonality property $(SO)$ then its norm
must be Euclidean. In particular, its subspace $(\mathbb{R}^{2},\|\cdot\|)$
must be Euclidean as well. Thus we are led to define the following.
\begin{defn}
[Stable Symmetric Orthogonality] A geodesic space $(M,d)$ is said
to satisfy the \emph{stable symmetric orthogonality} property $(SO^{*})$
if the metric space $M\times_{2}\mathbb{R}=(M\times\mathbb{R},\tilde{d})$
satisfies symmetric orthogonality property $(SO)$ where $\tilde{d}((x,t),(y,s))^{2}=d(x,y)^{2}+|t-s|^{2}$.
The local version will be denoted by $(SO^{*})_{loc}$.
\end{defn}
\begin{cor}
\label{cor:stableSO}Every normed vector space and every Finsler manifold
satisfying $(SO^{*})_{loc}$ is a Riemannian manifold.
\end{cor}
More generally, we can show that the stable symmetric orthogonality
$(SO^{*})$ implies the existence of symmetric angles if one-sided
angles are well-defined. This is the case if the metric is s\emph{trictly
$p$-convex}, $p>1$, i.e. if for any $x\in M$ the map $y\mapsto d^{p}(x,y)$
is strictly convex. 
\begin{prop}
\label{prop:two-sided-angles}Assume $(M,d)$ is strictly $p$-convex
for some $p>1$. Then $(M,d)$ satisfies the stable symmetric orthogonality
property $(SO^{*})$ if and only if for all unit speed geodesics $\gamma,\eta:[0,1]\to M$
starting at $\gamma_{0}=\eta_{0}$ it holds 
\[
\lim_{t\to0^{+}}\frac{d^{2}(\eta_{t},\gamma_{1})-d^{2}(\eta_{0},\gamma_{1})}{t}=\lim_{s\to0^{+}}\frac{d^{2}(\eta_{1},\gamma_{s})-d^{2}(\eta_{1},\gamma_{0})}{s}.
\]
\end{prop}
\begin{rem*}
This kind of symmetry condition was introduced in \cite{OhtaPalfia2014}
and has strong implications on the behavior of gradient flows of convex
functional. In particular, it shows that spaces satisfying $(SO^{*})$
must be Riemannian-like. Also note that this symmetry property is
stable under taking $L^{2}$-products justifying the terminology.
\end{rem*}
\begin{proof}
Note that $p$-convexity implies that $\partial_{s}^{+}d^{2}(\eta_{1},\gamma_{s})|_{s=0}$
and $\partial_{t}^{+}d^{2}(\eta_{t},\gamma_{1})|_{t=0}$ exist. 

Denote by $\tilde{d}$ the $L^{2}$-product metric on $M\times\mathbb{R}$.
Then there is an $a\in\mathbb{R}$ such that 
\[
\partial_{s}^{+}\tilde{d}^{2}((\eta_{1},1),(\gamma_{s},a\cdot s))|_{s=0}=0.
\]
 Since $M\times_{2}\mathbb{R}$ is also strictly $p$-convex \cite{Foertsch2004},
we see that the closest point of $(\eta_{1},1)$ onto $s\mapsto(\gamma_{s},a\cdot s)$
is $(\eta_{0},0)=(\gamma_{0},0)$. 

Since the orthogonality in $M\times_{2}\mathbb{R}$ is symmetric we
must have 
\[
\partial_{s}^{+}\tilde{d}^{2}((\eta_{t},t),(\gamma_{1},a))|_{s=0}\ge0
\]
 implying 
\[
\partial_{s}^{+}d^{2}(\eta_{1},\gamma_{s})|_{s=0}\le\partial_{t}^{+}d^{2}(\eta_{t},\gamma_{1})|_{t=0}.
\]
 Exchanging the roles of $\gamma$ and $\eta$ we obtain 
\[
\lim_{t\to0^{+}}\frac{d^{2}(\eta_{t},\gamma_{1})-d^{2}(\eta_{0},\gamma_{1})}{t}=\lim_{s\to0^{+}}\frac{d^{2}(\eta_{1},\gamma_{s})-d^{2}(\eta_{1},\gamma_{0})}{s}.
\]
Conversely, if the commutativity condition holds then it also holds
for $M\times_{2}\mathbb{R}$. Together with strict $p$-convexity
one sees that $(SO)$ holds for $M\times_{2}\mathbb{R}$. Hence $M$
satisfies $(SO^{*})$. 
\end{proof}
\begin{rem*}
[Jensen's inequality] In \cite{Kuwae2013} Kuwae proved Jensen's
inequality for spaces satisfying $(SO)$. However, it seems that the
proof of \cite[Theorem 4.1]{Kuwae2013} requires the stronger condition
$(SO^{*})$. Indeed, \cite[Lemma 2.12]{Kuwae2013} cannot hold in
general, because if an $L^{2}$-product of two non-trivial spaces
satisfies $(SO)$ then both of its factors have to satisfy property
$(SO^{*})_{loc}$. Furthermore, a general $L^{p}$-product of smooth
spaces with property $(SO)$ can have at most dimension two. Therefore,
for general $p\ne2$, Jensen's inequality on higher dimensional spaces
seems still open.
\end{rem*}
In the following we focus only on the global version of $(SO)$. If
$(M,d)$ is \emph{locally convex}, i.e. each point admits a convex
neighborhood, then almost all results below hold with respect to their
local version. However, local convexity seems rather strong as there
are spaces without convex sets with interior. 

The following lemma gives an equivalent characterization of the symmetric
orthogonality property in terms of a weak form of non-expansiveness
of projections onto geodesics. 
\begin{lem}
\label{lem:SOiffA}The condition $(SO)$ is equivalent to the following
property $(A)$: For all weakly convex sets $C$ it holds 
\[
d(x_{C},y)\le d(x,y)
\]
for all $x\in M$, $y\in C$ and $x_{C}\in\pi_{C}(x)$. 
\end{lem}
\begin{rem*}
In \cite[Lemma 2.10]{Kuwae2013} Kuwae proved that property $(SO)$,
property $(B)$ in that paper, implies property $(A)$.
\end{rem*}
\begin{proof}
Assume first the symmetric orthogonality property $(SO)$ holds and
let $C$ be a weakly convex subset. Choose $x_{C}\in\pi_{C}(x)$ and
$y\in C$. Let $\eta$ be a geodesic connecting $x_{C}$ and $y$
in $C$ and $\gamma$ be a geodesic connecting $x_{C}$ and $x$.
Since $\eta_{t}\in C$ and $x_{C}\in\pi_{C}(\gamma_{s})$ we have
\[
d(x_{C},\gamma_{s})=d(\eta_{0},\gamma_{s})\le d(\eta_{t},\gamma_{s}),
\]
i.e. $\gamma\bot_{x_{C}}\eta$. Then the symmetric orthogonality $(SO)$
implies $\eta\bot_{x_{C}}\gamma$ which is nothing but 
\[
d(x_{C},y)=d(\gamma_{0},y)\le d(\gamma_{1},y)=d(x,y).
\]
Conversely, assume for all weakly convex set $C$ it holds 
\[
d(x_{C},y)\le d(x,y)
\]
for all $x\in M$, $y\in C$ and $x_{C}\in\pi_{C}(x)$. Take now two
geodesics $\gamma$ and $\eta$ with $\gamma\bot_{p}\eta$. Note that
$C=\cup_{t\in[0,1]}\{\eta_{t}\}$ is weakly convex and $\gamma_{0}\in\pi_{C}(x)$.
Thus if $x=\gamma_{t}$ and $y=\eta_{1}\in C$ then 
\[
d(\eta_{0},y)=d(p_{0},y)\le d(x,y)=d(\gamma_{t},y)
\]
implying $\eta\bot_{p_{0}}\gamma$. As $\gamma$ and $\eta$ are arbitrary
$(M,d)$ must have symmetric orthogonalities. 
\end{proof}

\section*{Non-expansive projections}

In this section we introduce the non-expansive projection property
and prove Theorem \ref{thm:Main1-SOeqNE}.
\begin{defn}
[Non-expansive Projections] We say a geodesic space $(M,d)$ satisfies
the \emph{non-expansive projection property} $(NE)$ if for all closed
weakly convex sets $C$ and all $x,y\in M$ it holds 
\[
d(x_{C},y_{C})\le d(x,y)
\]
whenever $x_{C}\in\pi_{C}(x)$ and $y_{C}\in\pi_{C}(y)$.
\end{defn}
\begin{rem*}
Property $(NE)$ is well-known for linear spaces, see \cite[Theorem 3]{Kakutani1939}
and \cite[Theorem 5.2]{Phelps1957} where it is shown that the only
higher dimensional Banach spaces with non-expansive projections are
Hilbert spaces, compare also with Lemma \ref{lem:Birkhoff-rigidity}. 
\end{rem*}
It is easy to see that a set must be weakly convex if the a projection
onto it is non-expansive. The following properties can be shown from
the non-expansive projection property $(NE)$. We leave the details
to the interested reader.
\begin{lem}
Assume $(M,d)$ has non-expansive projections $(NE)$. Then the following
holds:
\begin{itemize}
\item $(M,d)$ is uniquely geodesic
\item the projection map $\pi_{C}$ onto weakly convex sets $C$ is at most
single-valued. 
\item closed balls are strictly convex. 
\end{itemize}
In particular, any weakly convex set is convex and $\pi_{C}$ can
be regarded as a non-expansive map whenever $C$ is compact.
\end{lem}
We first observe that the non-expansive projection property $(NE)$
is stronger than the symmetric orthogonality property $(SO)$.
\begin{prop}
\label{prop:CPtoSO}If $(M,d)$ satisfies $(NE)$ then it also satisfies
$(SO)$. 
\end{prop}
\begin{proof}
Assume the non-expansive projection property $(NE)$ holds and $\gamma\bot_{\gamma_{0}}\eta$,
i.e. $\gamma_{0}=\pi_{\eta}(\gamma_{s})$, $s\in[0,1]$. Since $\eta$
is (weakly) convex, by property $(NE)$ we have 
\[
d(\gamma_{0},\eta_{t})\le d(\gamma_{s},\eta_{t})\quad\text{for all }t\in[0,1].
\]
This implies that $\gamma_{0}=\pi_{\gamma}(\eta_{t})$ and hence $\eta\bot_{\gamma_{0}}\gamma$.
Because $\gamma$ and $\eta$ are arbitrary we see that the symmetric
orthogonality property $(SO)$ holds. 
\end{proof}
The converse of the statement does not hold, not even if every ball
is strictly convex. Indeed, if $(M,d)$ is a closed ball of radius
$R<\frac{\pi}{2}$ on the sphere $\mathbb{S}^{n}$ with standard metric
then its balls are strictly convex. However, the projection onto a
non-constant geodesic is never non-expansive.

Assuming Busemann convexity it is even possible to prove equivalence
of the two properties.
\begin{thm}
\label{thm:NEiffSO}A Busemann convex metric space has symmetric orthogonalities
$(SO)$ if and only if it has non-expansive projections $(NE)$. 
\end{thm}
\begin{rem*}
(1) As $M\times_{2}\mathbb{R}$ is Busemann convex spaces whenever
$M$ is Busemann convex, we also see that $(SO^{*})$ and $(NE^{*})$
are equivalent for Busemann convex spaces.

(2) An earlier version of this note also proved the equivalence of
$(SO)$ and $(NE)$ for Pedersen convex metric spaces, i.e. geodesic
spaces such that 
\[
\bigcup_{x\in C}\bar{B}_{\epsilon}(x)
\]
is weakly convex for all closed (weakly) convex sets $C$.
\end{rem*}
\begin{proof}
It suffices to show that $(SO)$ implies $(NE)$. Let $x,y\in M$
and $C$ be a closed convex sets such that $\pi_{C}(x)$ and $\pi_{C}(y)$
are non-empty. Let $x_{C}\in\pi_{C}(x)$ and $y_{C}=\pi_{C}(y)$ and
assume by exchanging $x$ and $y$ if necessary that $m=d(x,C)\le d(y,C)$. 

Set $C_{r}=\bar{B}_{r}(C)$ and note that $\pi_{C_{r}}(x)$ and $\pi_{C_{r}}(y)$
are non-empty as they contain points on the geodesics connecting $x$
and $x_{C}$ and resp. $y$ and $y_{C}$. Note that each $C_{r}$
is convex by Busemann convexity. Denote the projection of $x$ and
$y$ onto $C_{r}$ by $x_{r}$ and $y_{r}$, respectively. Since $x=x_{m}$
and $C_{m}$ is convex, the geodesic $\eta$ connecting $x$ and $y_{m}$
is in $C_{m}$. In particular, $\pi_{\eta}(y)=\pi_{C_{m}}(y)$. Now
property $(SO)$ (see Lemma \ref{lem:SOiffA}) implies
\[
d(x,y_{m})\le d(x,y).
\]
Replacing $y$ by $y_{m}$ we see that it suffices to show that $d(x_{C},y_{C})\le d(x,y)$
whenever $m=d(x,C)=d(y,C)$. 

Let $t\mapsto x_{t}$ and $t\mapsto y_{t}$ be $[0,m]$-parametrized
geodesics connecting $x_{C}$ and $x$, and $y_{C}$ and $y$, respectively.
Note that $x_{t}\in\pi_{C_{mt}}(x)$ and $y_{t}\in\pi_{C_{mt}}(y)$. 

Denote by $\gamma^{(t)}$ the geodesic connecting $x_{t}$ and $y_{t}$
for $t\in[0,1]$. Since 
\[
s\mapsto d(\gamma_{s}^{(t)},\gamma_{s}^{(0)})
\]
is convex by Busemann convexity and $\gamma^{(0)}$ in $C$ we see
that 
\[
d(\gamma_{s}^{(t)},C)\le(1-s)d(x_{t},C)+sd(y_{t},C)=mt\quad\mbox{for all }s\in[0,1].
\]
Furthermore, if 
\[
d(\gamma_{s}^{(t)},C)=mt\quad\mbox{for some }s\in[0,1]
\]
 then 
\[
s\mapsto d(\gamma_{s}^{(t)},\gamma_{s}^{(0)})=mt\quad\mbox{for all }s\in[0,1]
\]
 implying that the closed convex hull of $\gamma^{(t)}([0,1])\cup\gamma^{(0)}([0,1])$
is isometric to a closed convex set in $\mathbb{R}^{2}$ equipped
with a strictly convex norm, see Lemma \ref{lem:Busemann-convex-affine-rigidity}.
But in this two-dimensional setting $s\mapsto d(\gamma_{s}^{(t)},\gamma_{s}^{(0)})$
is constant if and only if $t\mapsto d(\gamma_{0}^{(t)},\gamma_{1}^{(t)})=d(x_{t},y_{t})$
is constant. In particular, $d(x_{C},y_{C})=d(x_{t'},x_{t'})$ for
all $t\in[0,t']$. 

Thus, replacing $C$ by $C_{t}$ we may assume that $m_{0}=d(\gamma_{\frac{1}{2}}^{(t)},C)<tm$
for all $t\in(0,1]$.  Observe by applying an argument as above to
the pairs $(z,x)$ and $(z,y)$ where $z=\gamma_{\frac{1}{2}}^{(1)}$
we obtain 
\begin{align*}
d(x_{m_{0}},z) & \le d(x,z)\\
d(y_{m_{0}},z) & \le d(y,z).
\end{align*}
Thus by triangle inequality 
\[
d(x_{m_{0}},y_{m_{0}})<d(x,y)
\]
where the strict inequality is due to the fact that $z$ cannot be
midpoint of $x_{m_{0}}$ and $y_{m_{0}}$. Replacing $(x,y)$ by $(x_{m_{0}},y_{m_{0}})$
we obtain inductively a $(m_{n})_{n\in\mathbb{N}}$ with $m_{n}=d(x_{m_{n}},C)$
and 
\[
\lim_{n\to\infty}m_{n}=\lim_{n\to\infty}d(z_{m_{n}},C)
\]
where $z_{t}$ is the midpoint of $x_{t}$ and $y_{t}$. By assumption
$d(z_{t},C)<m_{n}$ implying $(x_{m_{n}},x_{m_{n}})\to(x_{0},y_{0})=(x_{C},y_{C})$.
Thus we obtain the desire inequality 
\[
d(x_{C},y_{C})<d(x,y).
\]
\end{proof}

\section*{Busemann convex spaces with non-expansive projection property $(NE^{*})$}

In this section we are going to prove the second main result of this
note.
\begin{thm}
\label{thm:BusemannSO-CAT0}Let $(M,d)$ be a complete geodesic space
which is uniformly $\infty$-convex. Then the following are equivalent:

\begin{enumerate}
\item $(M,d)$ is a $CAT(0)$-space
\item $(M,d)$ is a Busemann convex space satisfying the stable non-expansive
projection property $(NE^{*})$
\item $(M,d)$ is a Busemann convex space satisfying the stable symmetric
orthogonality property $(SO^{*})$.
\end{enumerate}
\end{thm}
We first prove the following lemma on stability of the condition $(SO)$.
\begin{lem}
\label{lem:ultralimit-so}Let $(M_{n},d_{n},x_{n})_{n\in\mathbb{N}}$
be a sequence of geodesic spaces satisfying the symmetric orthogonality
$(SO)$ (resp. its stable version $(SO^{*})$). If the ultralimit
$\lim_{\omega}(M_{n},d_{n},x_{n})$ is uniquely geodesic and projections
onto compact convex sets are unique then it satisfies the symmetric
orthogonality property $(SO)$ (resp. its stable version $(SO^{*})$).
\end{lem}
\begin{proof}
The $*$-version follows by noting that 
\[
\left(\lim_{\omega}(M_{n},d_{n},x_{n})\times_{2}(\mathbb{R},|\cdot|,0)\right)=\left(\lim_{\omega}(M_{n},d_{n},x_{n})\right)\times_{2}(\mathbb{R},|\cdot|,0).
\]
Since $(M_{\omega},d_{\omega},x_{\omega})=\lim_{\omega}(M_{n},d_{n},x_{n})$
is uniquely geodesic, any geodesic in $(M_{\omega},d_{\omega})$ is
given by an ultralimit of a sequence geodesics $\gamma_{n}$. So it
suffices to show for geodesic $(\gamma_{n})$ and $(\eta_{n})$ in
$M_{\omega}$ with $(\gamma_{n})\bot_{(p_{n})}(\eta_{n})$ also $(\eta_{n})\bot_{(p_{n})}(\gamma_{n})$
holds. 

Assume $(\gamma_{n})\bot_{(p_{n})}(\eta_{n})$. Since projections
onto $(\eta_{n})([0,1])$ are unique and $p_{n}=\gamma_{n}(0)=\eta_{n}(0)$
we have 
\[
\lim_{\omega}d_{n}(p_{n},\gamma_{n}(t))<\lim_{\omega}d_{n}(\eta_{n}(s),\gamma_{n}(t))
\]
for all $s,t\in(0,1]$.

Fix $t\in(0,1]$ and let $q_{n}$ be a closest point of $\gamma_{n}(t)$
on $\eta_{n}$. Then $d_{n}(q_{n},\gamma_{n}(t))\le d_{n}(p_{n},\gamma_{n}(t))$
so that 
\[
\lim_{\omega}d_{n}(p_{n},\gamma_{n}(t))\le\lim_{\omega}d_{n}(q_{n},\gamma_{n}(t)).
\]
But the ultralimit of $(q_{n})$ is on the ultralimit of the geodesics
$(\eta_{n})$ implying that the ultralimits of $(p_{n})$ and $(q_{n})$
agree. 

Denote by $\tilde{\gamma}_{n}$ the geodesic connecting $q_{n}$ and
$\gamma_{n}(1)$. Then $\tilde{\gamma}_{n}\bot_{q_{n}}\eta_{n}$ so
that property $(SO)$ for $(M_{n},d_{n})$ implies $\eta_{n}\bot_{q_{n}}\tilde{\gamma}_{n}$.
In particular, since $\tilde{\gamma}_{n}(1)=\gamma_{n}(1)$ it holds
\[
d_{n}(\eta_{n}(s),q_{n})\le d_{n}(\eta_{n}(s),\gamma_{n}(1))\quad\mbox{for all }s\in[0,1].
\]

Combining the above we obtain for the ultralimit
\[
\lim_{\omega}d_{n}(\eta_{n}(s),p_{n})=\lim_{\omega}d_{n}(\eta_{n}(s),q_{n})\le\lim_{\omega}d_{n}(\eta_{n}(s),\gamma_{n}(t))
\]
for all $s\in[0,1]$. Since $t\in(0,1]$ is arbitrary we see that
$(\eta_{n})\bot_{p_{n}}(\gamma_{n})$. Thus property $(SO)$ holds
for $(M_{\omega},d_{\omega})$.
\end{proof}
\begin{lem}
\label{lem:ultralimit-uni-conv}Assume $(M_{n},d_{n},x_{n})$ is a
sequence of uniformly $\infty$-convex metric spaces with same uniformity
function $\rho$ then any ultralimit $\lim_{\omega}(M_{n},d_{n},x_{n})$
is uniformly $\infty$-convex with uniformity function $\rho$.
\end{lem}
\begin{proof}
Let $(x_{n}),(y_{n})$ and $(z_{n})$ be three points in $(M_{\omega},d_{\omega},x_{\omega})=\lim_{\omega}(M_{n},d_{n},x_{n})$
with 
\[
d_{\omega}((y_{n}),(z_{n}))>\epsilon\max\{d_{\omega}((x_{n}),(y_{n})),d_{\omega}((x_{n}),(z_{n}))\}.
\]
Assume w.l.o.g. $d_{\omega}((x_{n}),(y_{n}))\ge d_{\omega}((x_{n}),(z_{n}))$.
Then $\omega(A)=1$ for 
\[
A=\{n\in\mathbb{N}\,|\,d_{n}(y_{n},z_{n})>\epsilon d_{n}(x_{n},y_{n})\}.
\]
Now let $(w_{n})$ be a midpoint of $(y_{n})$ and $(z_{n})$ (w.r.t.
$d_{\omega}$). Note that $w_{n}$ may not be a midpoint $m_{n}$
of $y_{n}$ and $z_{n}$. 

We claim that $(w_{n})=(m_{n})$ in $M_{\omega}$. Assume by contradiction
this is not the case. Then $d_{\omega}((m_{n}),(w_{n}))>0$ and thus
$d_{\omega}((y_{n}),(z_{n}))>0$. So for some $\delta>0$ it holds
$\omega(B)=1$ where 
\begin{eqnarray*}
B & = & \{n\in A\,|\,d_{n}(m_{n},w_{n}),d_{n}(y_{n},z_{n})\ge\delta,\\
 &  & \qquad d_{n}(y_{n},w_{n}),d_{n}(z_{n},w_{n})\le\frac{1}{2}\left(d_{n}(y_{n},z_{n})+\rho_{0}\right)\}
\end{eqnarray*}
with 
\[
\rho_{0}=\frac{\rho(\nicefrac{\delta}{2})\delta}{2(1-\rho(\nicefrac{\delta}{2}))}>0.
\]
Note that we used the fact that $(w_{n})$ is a midpoint of $(y_{n})$
and $(z_{n})$.

Let $(v_{n})$ be the sequence of midpoints of $m_{n}$ and $w_{n}$.
Then by uniform convexity 
\begin{eqnarray*}
d_{n}(y_{n},v_{n}) & \le & (1-\rho(\nicefrac{\delta}{2}))\max\{d_{n}(y_{n},w_{n}),\frac{1}{2}d(y_{n},z_{n})\}\\
d_{n}(z_{n},v_{n}) & \le & (1-\rho(\nicefrac{\delta}{2}))\max\{d_{n}(z_{n},w_{n}),\frac{1}{2}d(y_{n},z_{n})\}
\end{eqnarray*}
for $n\in B$. But then 
\begin{eqnarray*}
d_{n}(y_{n},v_{n})+d_{n}(z_{n},v_{n}) & \le & \frac{1}{2}(1-\rho(\nicefrac{\delta}{2}))(d_{n}(y_{n},z_{n})+\rho_{0})\\
 & \le & d_{n}(y_{n},z_{n})-\rho(\nicefrac{\delta}{2})(d(y_{n},z_{n})-\frac{\delta}{2})\\
 & < & d_{n}(y_{n},z_{n})
\end{eqnarray*}
for $n\in B$ which contradicts the triangle inequality. Thus it holds
$(m_{n})=(w_{n})$. 

Then uniform convexity implies 
\[
d_{n}(x_{n},m_{n})\ge(1-\rho(\epsilon))d_{n}(x_{n},y_{n})
\]
for $n\in B$ so that 
\[
d_{\omega}((x_{n}),(w_{n}))\ge(1-\rho(\epsilon))d_{\omega}((x_{n}),(y_{n})).
\]
\end{proof}

\begin{proof}
[Proof of Theorem \ref{thm:BusemannSO-CAT0}] Let $(\lambda_{n})_{n\in\mathbb{N}}$
be sequence in $(0,1)$ with $\lambda_{n}\to0$ and choose any ultrafilter
$\omega$ on $\mathbb{N}$. In the following each tangent cone $(T_{x}^{(o)}M,d_{x})$
will denote the ultralimit (with respect to $\omega$) of the space
pointed metric space $(M,d_{n},x)$ where $d_{n}=\lambda_{n}^{-1}d$. 

The assumptions of the theorem imply that each tangent cone $(T_{x}^{(o)}M,d_{x})$
is uniformly $\infty$-convex and hence uniquely geodesic and has
single-valued projections onto closed convex sets. In particular,
they are Busemann convex and satisfy the stable symmetric orthogonality
property $(SO)$. 

Choose any triple $x,y,z\in M$. Let $m$ be the midpoint of $x$
and $y$, $\gamma:[-1,1]\to M$ be the geodesic between $x$ and $y$
and $\eta:[0,1]\to M$ be the geodesic between $m$ and $z$. Define
\begin{align*}
m_{n} & =m\\
x_{n} & =\gamma_{-\lambda_{n}}\\
y_{n} & =\gamma_{\lambda_{n}}\\
z_{n} & =\eta_{\lambda_{n}}.
\end{align*}
Denote the ultralimits of the sequences in $T_{m}^{(o)}M$ by $m_{\infty}$,
$x_{\infty}$, $y_{\infty}$ and $z_{\infty}$ respectively. 

From Busemann convexity and the properties of the geodesics $\gamma$
and $\eta$ we have 
\begin{align*}
\frac{1}{2}d_{n}(x_{n},z_{n})^{2}+\frac{1}{2}d_{n}(y_{n},z_{n})^{2}-d_{n}(m_{n},z_{n})^{2}-\frac{1}{4}d_{n}(x_{n},y_{n})^{2}\\
\le\frac{1}{2}d(x,z)^{2}+\frac{1}{2}d(y,z)^{2}-d(m,z)^{2}-\frac{1}{4}d(x,y)^{2}
\end{align*}
which implies 
\begin{align*}
\frac{1}{2}d_{m}(x_{\infty},z_{\infty})^{2}+\frac{1}{2}d_{m}(y_{\infty},z_{\infty})^{2}-d_{m}(m_{\infty},z_{\infty})^{2}-\frac{1}{4}d_{m}(x_{\infty},y_{\infty})^{2}\\
\le\frac{1}{2}d(x,z)^{2}+\frac{1}{2}d(y,z)^{2}-d(m,z)^{2}-\frac{1}{4}d(x,y)^{2}.
\end{align*}
Observe that the geodesics $\gamma$ and $\eta$ induce an ultraline
$\bar{\gamma}:\mathbb{R}\to T_{m}^{(o)}M$ and an ultraray $\bar{\eta}:[0,\infty)\to T_{m}^{(o)}M$.
Note that $x_{\infty}$, $y_{\infty}$ and $m_{\infty}$ lie on $\bar{\gamma}$
and $z_{\infty}$ and $m_{\infty}$ lie on $\bar{\eta}$. 

Let $\bar{\gamma}^{\pm}$ be the two ultrarays obtained from $\gamma$.
If $\bar{\gamma}^{+}=\bar{\eta}$ or $\bar{\gamma}^{-}=\bar{\eta}$
then $x_{\infty}$, $y_{\infty}$, $m_{\infty}$ and $z_{\infty}$
all lie on $\bar{\gamma}$ so that the one-dimensional parallelogram
identity yields 
\[
\frac{1}{2}d_{m}(x_{\infty},z_{\infty})^{2}+\frac{1}{2}d_{m}(y_{\infty},z_{\infty})^{2}-d_{m}(m_{\infty},z_{\infty})^{2}-\frac{1}{4}d_{m}(x_{\infty},y_{\infty})^{2}=0.
\]

Assume the ray $\bar{\eta}$ is distinct from the rays $\bar{\gamma}^{\pm}$.
Then by Lemma \ref{lem:flat-sector} the convex hulls $C^{\pm}$ of
$\bar{\gamma}^{\pm}([0,\infty))\cup\bar{\eta}([0,\infty))$ are both
flat non-trivial sectors. As both $C^{+}$ and $C^{-}$ are closed
convex subsets of $(T_{m}^{(o)}M,d_{m})$, the geodesic spaces $(C^{\pm},d_{m})$
satisfy the stable symmetric orthogonality property, so that by Corollary
\ref{cor:stableSO} $(C^{\pm},d_{m})$ are both flat Euclidean sectors.
But then Lemma \ref{lem:Euclidean-double-sector} implies that 
\[
\frac{1}{2}d_{m}(x_{\infty},z_{\infty})^{2}+\frac{1}{2}d_{m}(y_{\infty},z_{\infty})^{2}-d_{m}(m_{\infty},z_{\infty})^{2}-\frac{1}{4}d_{m}(x_{\infty},y_{\infty})^{2}\ge0.
\]
 Combined with the inequality above we have shown that 
\[
\frac{1}{2}d(x,z)^{2}+\frac{1}{2}d(y,z)^{2}-d(m,z)^{2}-\frac{1}{4}d(x,y)^{2}\ge0.
\]
As $x,y,z\in M$ are arbitrary, $(M,d)$ must be a $CAT(0)$-space
proving the claim of the theorem.
\end{proof}
\begin{rem*}
(1) A previous version of this note used the result in \cite{Foertsch2010}. 

(2) The proof uses the following observation of Busemann convex spaces:
If $(M,d)$ is Busemann convex and each point has a tangent cone which
is a $CAT(0)$-space then $(M,d)$ itself is a $CAT(0)$-space. A
similar argument holds for non-negatively curved spaces in the sense
of \cite{Kell2016a}, i.e. if $(M,d)$ is Busemann concave and each
tangent cone is non-negatively curved in the sense of Alexandrov then
$(M,d)$ is non-negatively curved in the sense of Alexandrov.
\end{rem*}

\subsection*{Generalizations}

The assumption of uniform $\infty$-convexity can be dropped if it
is possible to show the following.
\begin{problem}
Assume $(M,d)$ is Busemann convex and $\gamma$ and $\eta$ are two
geodesics starting at $x$. Let $\bar{\gamma}$ and $\bar{\eta}$
the corresponding ultrarays in $T_{x}^{(o)}M$. Then there is a (weakly
convex) $2$-dimensional flat sector $C$ containing $\bar{\gamma}$
and $\bar{\eta}$.
\end{problem}
Indeed, $C\times_{2}\mathbb{R}$ is often convex along the ultralimits
of geodesics in $M\times_{2}\mathbb{R}$ (see \cite{Kleiner1999,Foertsch2010})
and the proof of Lemma \ref{lem:ultralimit-so} shows that $(SO)$
holds for those geodesics. Since Lemma \ref{lem:Birkhoff-rigidity}
(see \cite{James1947}) only needs $(SO)$ for the straight lines
it follows that $C\times_{2}\mathbb{R}$ is Euclidean. 

\bibliographystyle{amsalpha}
\bibliography{bib2}

\end{document}